\documentclass[12pt,a4paper,UKenglish]{article}
\usepackage[margin=3cm]{geometry}
\usepackage{amsthm} 
\usepackage{amssymb}
\usepackage{amsopn}
\usepackage{color}
\usepackage{babel} 
\usepackage{hyperref}
\usepackage{amsmath, amsfonts, amssymb} 
\usepackage{graphicx}
\allowdisplaybreaks 
\usepackage[utf8]{inputenc}
\usepackage[numbers,sort&compress]{natbib}

\usepackage{tikz}
\usetikzlibrary{topaths,calc}

\usepackage{amssymb, amsmath}
\usepackage{verbatim}
\usepackage{hyperref}

\theoremstyle{plain}
\newtheorem{theorem}{Theorem}[section]
\newtheorem{lemma}[theorem]{Lemma}
\newtheorem{corollary}[theorem]{Corollary}
\newtheorem{proposition}[theorem]{Proposition}

\theoremstyle{definition}
\newtheorem{definition}[theorem]{Definition}
\newtheorem{example}[theorem]{Example}

\theoremstyle{remark}

\newtheorem{remark}[theorem]{Remark}

\usepackage[affil-it]{authblk}

\begin{document}
	\title{Normalized Laplace Operators for Hypergraphs with Real Coefficients}
	
	\author[1,2]{Jürgen Jost}
	\author[1,3,4]{Raffaella Mulas\footnote{Email address: r.mulas@soton.ac.uk}}
	\affil[1]{Max Planck Institute for Mathematics in the Sciences, Leipzig, Germany}
	\affil[2]{Santa Fe Institute for the Sciences of Complexity, New Mexico, USA}
	\affil[3]{The Alan Turing Institute, London, UK}
	\affil[4]{University of Southampton, Southampton, UK}
	\date{}

	\maketitle
	
	\begin{abstract}Chemical hypergraphs and their associated normalized Laplace operators are generalized and studied in the case where each vertex--hyperedge incidence has a real coefficient. We systematically study the effect of symmetries of a hypergraph on the spectrum of the Laplacian.
		\vspace{0.2cm}
		
		\noindent {\bf Keywords:} Chemical hypergraphs, Laplace operator, Spectral theory
	\end{abstract}
	\section{Introduction}
	The essential structure of many empirical systems consists in pairwise interactions between their elements, and such network structures  are therefore mathematically as (possibly weighted and/or directed) graphs. The theory of such networks is well developed, and it has proven successful in many applications, see for instance \cite{Estrada12}. Many other systems, however, naturally support higher order interactions. In particular, many social networks involve interactions between groups of varying size. Scientific collaboration networks, where data are readily available from preprint repositories, are a standard example. Also, chemical reactions typically involve several educts, several products and perhaps some catalysts. Such structures can no longer be modelled by graphs, but should rather be considered as hypergraphs.  In a hypergraph, a hyperedge can connect several vertices, and if we refine the structure a bit, we can in addition distinguish two classes of vertices, like the educts and products in a chemical reaction. While for a long time, hypergraphs have received less attention than ordinary graphs, this has recently changed, and the mathematical study of hypergraphs and its application to empirical systems has become very active, see for instance \cite{Battiston20}. As for the qualitative analysis of graphs (see for instance \cite{Banerjee08a}), the investigation of the spectrum of a Laplace operator can reveal important structural properties. In this contribution, we want to refine those tools, by looking at hypergraphs with real (and possibly negative) coefficients and studying the effects of symmetries on the spectrum. That is, we shall ask to what extent we can get indications of  the presence (or absence) of symmetries by looking at the spectrum of a hypergraph. 
	
	Let us now describe the setting of this paper in more concrete terms. In \cite{JM2019} we introduced \emph{chemical hypergraphs} with the aim of modelling chemical reaction networks. While classical hypergraphs are given by \emph{vertices} and sets of vertices called \emph{hyperedges}, chemical hypergraphs have the additional structure that each vertex has either a plus sign, or a minus sign, or both, for each hyperedge in which it is contained. The idea is that each hyperedge represents a chemical reaction involving the elements that it contains as vertices; plus signs are assigned to the \emph{educts} or \emph{inputs} of the reaction, minus signs are given to its \emph{products} or \emph{outputs}, and both signs are given to the \emph{catalysts} participating in the reaction. We also introduced two normalized Laplace operators for chemical hypergraphs, the \emph{vertex Laplacian} $L$ and the \emph{hyperedge Laplacian} $L^H$, as a generalization of the classical theory introduced by F. Chung \cite{Chung} for graphs. The spectra of these two operators, that only differ from each other by the multiplicity of the eigenvalue $0$, have been largely studied in \cite{Master-Stability,Sharp,MulasZhang,AndreottiMulas,pLaplacians1,spectralclasses,Independence,Symmetries}.
	
	Similarly, N. Reff and L. Rusnak \cite{ReffRusnak} introduced \emph{oriented hypergraphs}, i.e., hypergraphs in which each vertex in a hyperedge has either a plus or a minus sign, as well as a \emph{multiplicity}: a natural number assigned to the given vertex--hyperedge incidence. They also developed a spectral theory of oriented hypergraphs that focuses on the spectra of the (unnormalized) Laplace operator $\Delta$ and the adjacency matrix $A$. This theory has also been widely studied (see for instance \cite{ReffRusnak,hyp2013,hyp2014,hyp2015,hyp2016,hyp2017,hyp2018,hyp2018-2,hyp2019,hyp2019-2,hyp2019-3,hyp2020,hyp2020-2,Synchronization}) and, as pointed out in \cite[Remark 2.17]{MulasZhang}, for chemical hypergraphs without catalysts which are \emph{regular}, meaning that all vertices have the same degree, the spectra of $L$, $\Delta$ and $A$ are all equivalent, up to a multiplicative or additive constant. For non-regular chemical hypergraphs, however, the spectra of these operators are in general not equivalent, implying that they encode different properties of the hypergraph to which they are associated.
	
	Here we propose a further generalization of chemical hypergraphs and their associated operators, and we focus on the study of the normalized Laplacians. In particular, we consider hypergraphs in which each vertex--hyperedge incidence has a \emph{real coefficient}.\newline
	
	\textbf{Structure of the paper.} In Section \ref{Section Basic definitions} we establish the basic definitions for hypergraphs with real coefficients and their associated operators, as a natural generalization of the constructions in \cite{JM2019}. In Section \ref{Section First properties} we investigate the first properties of the generalized normalized Laplacians and in Section \ref{Section The eigenvalue 0} we discuss the multiplicity of the eigenvalue $0$. In Section \ref{Section min-max principle} we apply the min-max principle in order to characterize all the eigenvalues in terms of the Rayleigh Quotient, and in Section \ref{Section Largest eigenvalue} we investigate the largest eigenvalue. Finally, in Section \ref{Section:general} we give some general bounds on all eigenvalues; in Section \ref{Section Symmetries} and Section \ref{section: general symm} we study the effect of symmetries on the spectra of the generalized normalized Laplacians.
	
	\section{Basic definitions}\label{Section Basic definitions}
	\begin{definition}
		A \emph{hypergraph with real coefficients} is a triple $\Gamma=(V,H,\mathcal{C})$ such that:
		\begin{itemize}
			\item $V=\{v_1,\ldots,v_N\}$ is a finite set of \emph{nodes} or \emph{vertices};
			\item $H=\{h_1,\ldots,h_M\}$ is a multiset of elements $h_j\in \mathcal{P}(V)$ called \emph{hyperedges};
			\item $\mathcal{C}=\{C_{v,h}: v\in V\text{ and }h\in H \}$ is a set of \emph{coefficients} $C_{v,h}\in\mathbb{R}$ and it's such that
			\begin{equation}\label{eq:zerocoeff}
			C_{v,h}=0 \iff v\notin h.
			\end{equation}
		\end{itemize}
	\end{definition}
	\begin{remark}
		The chemical hypergraphs in \cite{JM2019} have coefficients $C_{v,h}\in \{-1,0,+1\}$ without the assumption \eqref{eq:zerocoeff}, while the oriented hypergraphs defined by Reff and Rusnak in \cite{ReffRusnak} can be seen as hypergraphs with coefficients $C_{v,h}\in \mathbb{Z}$. Moreover, a \emph{signed graph} is a hypergraph with real coefficients such that:
		\begin{itemize}
			\item[-] $H$ is a set (that is, $j\neq k$ implies $h_j\neq h_k$);
			\item[-] Each hyperedge contains exactly two vertices;
			\item[-] $C_{v,h}\in\{+1,-1\}$ for all $h\in H$ and $v\in h$.
		\end{itemize}A \emph{simple graph} is a signed graph such that, for each hyperedge $h$, there exists a unique $v\in h$ such that $C_{v,h}=1$ and there exists a unique $w\in h$ such that $C_{w,h}=-1$.
	\end{remark}
	
	\begin{remark}
		In the case when, for all $v$ and $h$, $C_{v,h}\geq 0$ and $\sum_{h\in H}C_{v,h}=1$, we can see each coefficient $C_{v,h}$ as the \emph{probability} of the vertex $v$ to belong to the hyperedge $h$. In the case of a hypergraph with integer coefficients, we can see each vertex as a chemical element, each hyperedge as a chemical reaction and each coefficient $C_{v,h}$ as the chemical \emph{stoichiometric coefficient} of the element $v$ in the reaction $h$.
	\end{remark}
	
	\begin{definition}
		We say that $\Gamma=(V,H,\mathcal{C})$ is \textit{connected} if, for every pair of vertices $v,w\in V$, there exists a path that connects $v$ and $w$, i.e. there exist $v_1,\dots,v_m\in V$ and $h_1,\dots,h_{m-1}\in H$ such that:
		\begin{itemize}
			\item $v_1=v$;
			\item $v_m=w$;
			\item $\{v_i,v_{i+1}\}\subseteq h_i$ for each $i=1,\dots,m-1$.
		\end{itemize}
	\end{definition}
	
	From here on, we fix a hypergraph with real coefficients $\Gamma=(V,H,\mathcal{C})$ on $N$ vertices $v_1,\ldots,v_N$ and $M$ hyperedges $h_1,\ldots,h_M$. For simplicity, assume that $\Gamma$ has no isolated vertices, i.e. we assume that each vertex is contained in at least one hyperedge, and we assume that $\Gamma$ is connected.
	
	Changing the orientation of a hyperedge $h$ means substituting every $C_{v,h}$ by $-C_{v,h}$. We denote the two orientations by $(h,+)$ and $(h,-)$. Analogously to differential forms in Riemannian geometry \cite{JGeom} and to the case of chemical hypergraphs \cite{JM2019}, we shall consider functions $\gamma:H\rightarrow\mathbb{R}$ that satisfy
	\begin{equation}
	\label{or}
	\gamma (h,-)=-\gamma (h,+),
	\end{equation}
	that is, changing the orientation of $h$ produces a minus sign.
	\begin{definition}
		Given $h\in H$, its \emph{cardinality}, denoted $|h|$, is the number of vertices that are contained in $h$.
	\end{definition}
	\begin{definition}
		Given $v\in V$, its \emph{degree} is
		\begin{equation}\label{eq:defdegree}
		\deg v:=\sum_{h\in H}(C_{v,h})^2.
		\end{equation}
		The $N\times N$ diagonal \emph{degree matrix} of $\Gamma$ is
		\begin{equation*}
		D:=\textrm{diag}\bigl(\deg v_i\bigr)_{i=1,\ldots,N}.
		\end{equation*}
	\end{definition}
	\begin{remark}In literature, there are various ways of defining the degree of a vertex in a hypergraph. Our definition of vertex degree coincides with the one in \cite{Sharp} for chemical hypergraphs and it coincides, in particular, with the classical definition of degree in the case of graphs. Since it is natural to choose a vertex degree that is always non-negative, one could as well consider, for instance, $\sum_{h\in H}|C_{v,h}|$ as degree of a vertex $v$. However, here we choose to define the degree of $v$ as in \eqref{eq:defdegree} in order to be able, later on, to write the vertex normalized Laplacian in a matrix formulation that coincides with the one that is commonly used in literature for graphs (cf. Proposition \ref{prop:matrix} below).
		
		Observe also that, since the coefficients $C_{v,h}$ are non-zero provided $v\in h$ and since we are assuming that there are no isolated vertices, there are no vertices of degree $0$. In particular, $D$ is an invertible matrix.
	\end{remark}
	\begin{definition}
		The $N\times N$ \emph{adjacency matrix} of $\Gamma$ is $A:=(A_{ij})_{ij}$, where $A_{ii}:=0$ for all $i=1,\ldots,N$ and
		\begin{equation*}
		A_{ij}:=-\sum_{h\in H}C_{v_i,h}\cdot C_{v_j,h}\quad \text{for all }i\neq j.
		\end{equation*}
	\end{definition}
	\begin{remark}
		In the case of simple graphs, $A$ coincides with the classical adjacency matrix that has $(0,1)$--entries and it's such that $A_{i,j}=1$ if and only if $v_i$ and $v_j$ are joined by an edge.
	\end{remark}
	
	\begin{definition}
		The $N\times M$ \emph{incidence matrix} of $\Gamma$ is $\mathcal{I}:=(\mathcal{I}_{ij})_{ij}$, where
		\begin{equation*} 
		\mathcal{I}_{ij}:=C_{v_i,h_j}.
		\end{equation*}
	\end{definition}
	\begin{remark}
		Each row $\mathcal{I}_i$ of $\mathcal{I}$ represents a vertex $v_i$ and each column $\mathcal{I}^j$ of $\mathcal{I}$ represents a hyperedge $h_j$.
	\end{remark}
	\begin{definition}
		Given $J\subseteq\{1,\ldots,M\}$, we say that the hyperedges $\{h_j\}_{j\in J}$ are \emph{linearly independent} if the corresponding columns in the incidence matrix are linearly independent, that is, if $\{\mathcal{I}^j\}_{j\in J}$ are linearly independent.
	\end{definition}
	\begin{remark}
		Linear dependence of hyperedges means the following: we see each hyperedge as the weighted sum of all its vertices, where the weights are the coefficients. If a hyperedge can be written as a linear combination of the other ones, we talk about linear dependence.
	\end{remark}

	\begin{definition}
		Given $\omega,\gamma:H\rightarrow\mathbb{R}$, their \emph{scalar product} is
		\begin{equation*}
		(\omega,\gamma)_H:=\sum_{h\in H}\omega(h)\cdot\gamma(h).
		\end{equation*}
	\end{definition}
	\begin{definition}
		Given $f,g:V\rightarrow\mathbb{R}$, their \emph{scalar product} is
		\begin{equation*}
		(f,g)_V:=\sum_{v\in V}\deg v\cdot f(v)\cdot g(v).
		\end{equation*}
	\end{definition}
	\begin{remark}
		The scalar products defined above are both positive definite.
	\end{remark}
	\begin{definition}
		Given $f:V\rightarrow\mathbb{R}$ and $h\in H$, its \emph{boundary operator} is
		\begin{equation*}
		\delta f(h):=\sum_{v\in V}C_{v,h}\cdot f(v).
		\end{equation*}\end{definition}
	\begin{remark}
		Observe that \begin{equation*}
		\delta:\{f:V\rightarrow\mathbb{R}\}\longrightarrow\{\gamma:H\rightarrow\mathbb{R}\}
		\end{equation*}
		and, for each $f:V\rightarrow\mathbb{R}$, $\delta f$ satisfies \eqref{or}.
	\end{remark}
	\begin{definition}
		Let 	\begin{equation*}
		\delta^*:\{\gamma:H\rightarrow\mathbb{R}\}\longrightarrow\{f:V\rightarrow\mathbb{R}\}
		\end{equation*}be defined by
		\begin{equation*}
		\delta^*(\gamma)(v):=\frac{\sum_{h\in H}C_{v,h}\cdot\gamma(h)}{\deg v}.
		\end{equation*}
	\end{definition}
	\begin{lemma}$\delta^*$ is such that $(\delta f,\gamma)_H=(f,\delta^*\gamma)_V$, therefore it is the (unique) adjoint operator of $\delta$.
	\end{lemma}
	\begin{proof}
		\begin{align*}
		(\delta f,\gamma)_H&=\sum_{h\in H}\gamma(h)\cdot\biggl(\sum_{v\in V}C_{v,h}\cdot f(v)\biggr)\\
		&=\sum_{v\in V}f(v)\cdot\biggl(\sum_{h\in H}C_{v,h}\cdot\gamma(h)\biggr)\\
		&=\sum_{v\in V}\deg v\cdot f(v)\cdot\frac{\biggl(\sum_{h\in H}C_{v,h}\cdot\gamma(h)\biggr)}{\deg v}\\
		&=\sum_{v\in V}\deg v\cdot f(v)\cdot\delta^*(\gamma)(v)\\
		&=(f,\delta^*\gamma)_V.
		\end{align*}
	\end{proof}
	\begin{definition}
		Given $f:V\rightarrow\mathbb{R}$ and given $v\in V$, let
		\begin{align*}
		Lf(v):&=\delta^*(\delta f)(v)\\
		&=\frac{\sum_{h\in H}C_{v,h}\cdot\delta f(h)}{\deg v}\\
		&=\frac{\sum_{h\in H}C_{v,h}\cdot\biggl(\sum_{v'\in V}C_{v',h}\cdot f(v')\biggr)}{\deg v}\\
		&=\frac{1}{\deg v}\cdot \biggl(\sum_{h\in H}\sum_{v'\in V}C_{v,h}\cdot C_{v',h}\cdot f(v')\biggr)\\
		&=f(v)+\frac{1}{\deg v}\cdot \biggl(\sum_{h\in H}\sum_{v'\in V\setminus\{v\}}C_{v,h}\cdot C_{v',h}\cdot f(v')\biggr).
		\end{align*}Analogously, given $\gamma:H\rightarrow\mathbb{R}$ and $h\in H$, let
		\begin{align*}
		L^H\gamma(h):&=\delta(\delta^* \gamma)(h)\\
		&=\sum_{v\in V}C_{v,h}\cdot \delta^* \gamma(v)\\
		&=\sum_{v\in V}C_{v,h}\cdot \biggl(\frac{\sum_{h'\in H}C_{v,h'}\cdot\gamma(h')}{\deg v}\biggr)\\
		&=\sum_{v\in V}\sum_{h'\in H}\frac{1}{\deg v}\cdot C_{v,h}\cdot C_{v,h'}\cdot \gamma(h').
		\end{align*}We call $L$ and $L^H$ the \emph{vertex normalized Laplacian} and the \emph{hyperedge normalized Laplacian}, respectively. They coincide with the ones in \cite{JM2019} in the case when the coefficients are in $\{-1,0,+1\}$.
	\end{definition}
	
	\begin{remark}
		Note that the operators $L$ and $L^H$ do not change if we change the signs of all coefficients $C_{v,h}$ for any hyperedge $h$ and all vertices $v$ in that hyperedge. Moreover, the operator $L^H$ doesn't change if, for each $v\in V$ and each $h\in H$, we replace $C_{v,h}$ by $C_{v,h}\cdot a(v)$, where $a(v)$ is any real coefficient depending only on $v$.
	\end{remark}
	\section{First properties}\label{Section First properties}
	We first observe that, as in the case of chemical hypergraphs \cite{JM2019},
	\begin{itemize}
		\item $L$ and $L^H$ are the two compositions of $\delta$ and $\delta^*$, which are adjoint to each other. Therefore they are both self-adjoint, which implies that their eigenvalues are real.
		\item $L$ and $L^H$ are non-negative operators. In fact, given $f:V\rightarrow\mathbb{R}$,
		\begin{equation}\label{pos1}
		\langle Lf,f\rangle=\langle \delta^*\delta f,f\rangle=\langle \delta f,\delta f\rangle_E\geq 0.
		\end{equation}Analogously, for $\gamma:E\rightarrow\mathbb{R}$,
		\begin{equation}\label{pos2}
		\langle L^H\gamma,\gamma\rangle_E=\langle\delta\delta^* \gamma,\gamma\rangle_E=\langle\delta^* \gamma,\delta^*\gamma\rangle\geq 0.
		\end{equation}This implies that the eigenvalues of $L$ and $L^H$ are non-negative.
		\item Since $L$ and $L^H$ are the two compositions of two linear operators, the non-zero eigenvalues of $L$ and $L^H$ are the same. In particular, if $f$ is an eigenfunction of $L$ with eigenvalue $\lambda \neq 0$, then $\delta f$ is an eigenfunction of $L^H$ with eigenvalue $\lambda$; if $\gamma$ is an eigenfunction of $L^H$ with eigenvalue $\lambda'\neq 0$, then $\delta^* \gamma$ is an eigenfunction of $L$ with eigenvalue $\lambda'$.
		
		This also implies that the two operators only differ in the multiplicity of the eigenvalue $0$. Let $m_V$ and $m_H$ be the multiplicity of the eigenvalue $0$ of $L$ and $L^H$, respectively. Then,
		\begin{equation*}
		m_V-m_H= N-M.
		\end{equation*}
	\end{itemize}
	
	Now, observe that we can see a function $f:V\rightarrow\mathbb{R}$ as a vector $(f_1,\ldots,f_N)\in\mathbb{R}^N$ such that $f_i=f(v_i)$ and, analogously, we can see a function $\gamma:H\rightarrow\mathbb{R}$ as a vector $(\gamma_1,\ldots,\gamma_M)\in\mathbb{R}^M$ such that $\gamma_k=\gamma(h_k)$. In view of this observation, we can write $L$ and $L^H$ in a matrix form.\newline
	
	In the following, we denote by $\mathrm{Id}$ the $N\times N$ identity matrix.
	
	\begin{proposition}\label{prop:matrix}
		The normalized Laplacians can be rewritten in matrix form as\begin{equation*}
		L=\mathrm{Id}-D^{-1}A \qquad \text{and}\qquad  L^H=\mathcal{I}^\top D^{-1}\mathcal{I}.
		\end{equation*}
	\end{proposition}
	\begin{proof}
		Observe that the $(ij)$--entry of the $N\times N$ matrix $\mathrm{Id}-D^{-1}A$ is
		\begin{equation*}
		\bigr(\mathrm{Id}-D^{-1}A)_{ij}=\begin{cases}1 & \text{if }i=j\\
		-\frac{A_{ij}}{\deg v_i} &\text{if }i\neq j.
		\end{cases}
		\end{equation*}Therefore, given $f:V\rightarrow\mathbb{R}$ that can be seen as a vector $f=(f_1,\ldots,f_N)\in\mathbb{R}^N$,
		\begin{align*}
		\biggr(\bigr(\mathrm{Id}-D^{-1}A\bigl)\cdot f\biggl)_i&=f_i-\frac{1}{\deg v_i}\cdot\sum_{j\neq i,j=1}^N A_{ij}\cdot f_j\\
		&=f(v_i)+\frac{1}{\deg v_i}\cdot \biggl(\sum_{h\in H}\sum_{v_j\in V\setminus\{v_i\}}C_{v_i,h}\cdot C_{v_j,h}\cdot f(v_j)\biggr) \\
		&=Lf(v_i).
		\end{align*}

		Similarly, the $(kl)$--entry of the $M\times M$ matrix $\mathcal{I}^\top D^{-1}\mathcal{I}$ is
		\begin{equation*}
		\bigr(\mathcal{I}^\top D^{-1}\mathcal{I}\bigl)_{kl}=\sum_{i=1}^N\frac{1}{\deg v_i}\cdot \mathcal{I}_{ik}\cdot \mathcal{I}_{il}.
		\end{equation*}Therefore, given $\gamma:H\rightarrow\mathbb{R}$ that can be seen as a vector $\gamma=(\gamma_1,\ldots,\gamma_M)\in\mathbb{R}^M$,
		\begin{align*}
		\biggr(\bigr(\mathcal{I}^\top D^{-1}\mathcal{I}\bigl)\cdot \gamma\biggl)_k&=\sum_{i=1}^N\sum_{l=1}^M\frac{1}{\deg v_i}\cdot \mathcal{I}_{ik}\mathcal{I}_{il}\cdot \gamma_l\\
		&=\sum_{v_i\in V}\sum_{h_l\in H}\frac{1}{\deg v_i}\cdot C_{v_i,h_k}\cdot C_{v_i,h_l}\cdot \gamma(h_l)\\
		&=L^H\gamma(h_k).
		\end{align*}
		
	\end{proof}
	\begin{corollary}
		The sum of the eigenvalues of $L$ (and $L^H$) equals $N$.
	\end{corollary}
	\begin{proof}
		It follows from Proposition \ref{prop:matrix}, since the sum of a matrix's eigenvalues equals its trace, and the trace of $L$ is $N$. 
	\end{proof}
	From here on, we will arrange the $N$ eigenvalues of $L$ as
	\begin{equation*}
	0\leq \lambda_1\leq \ldots \leq \lambda_N\leq N
	\end{equation*}and we will arrange the $M$ eigenvalues of $L^H$ as
	\begin{equation*}
	0\leq \lambda^H_1\leq \ldots \leq \lambda^H_M\leq N.
	\end{equation*}We say that the eigenvalues of $L$ are the \emph{spectrum} of $\Gamma$.
	\begin{remark}
		Assuming the connectivity of $\Gamma$ is not restrictive. In fact, it is clear by definition of $L$ that the spectrum of a disconnected hypergraph is equal to the union of the spectra of its connected components.
	\end{remark}
	
	\section{The eigenvalue 0}\label{Section The eigenvalue 0}
	By \eqref{pos1}, a function $f$ on the vertex set satisfies $L f=0$  if and only if, for every $h\in H$,
	\begin{equation}\label{eigenf0}\sum_{v\in V}C_{v,h}\cdot f(v)=0.\end{equation}
	Thus, to create an eigenvalue $0$ of $L$, we need a function  $f:V\rightarrow \mathbb{R}$ such that is not identically $0$ and satisfies \eqref{eigenf0}.
	
	Similarly, by \eqref{pos2}, in order to get an eigenvalue $0$ of $L^H$, we need $\gamma: H \to \mathbb{R}$ that is not identically $0$, satisfying \eqref{or} and
	\begin{equation}\label{mult0LH}
	\sum_{h\in H}C_{v,h}\cdot\gamma(h)=0
	\end{equation}
	for every vertex $v$.\newline
	
	The following results generalize the ones in \cite[Section 5]{JM2019}.
	
	\begin{proposition}\label{propker}
		\begin{equation*}
		\dim(\ker \mathcal{I})=m_H \qquad \text{ and }\qquad\dim(\ker \mathcal{I}^\top)=m_V.
		\end{equation*}
	\end{proposition}
	\begin{proof}
		Fix a function $\gamma:H\rightarrow\mathbb{R}$ that we can see as a vector $(\gamma_1,\ldots,\gamma_M)\in\mathbb{R}^M$. Then,
		\begin{align*}
		\mathcal{I}\cdot\begin{pmatrix} \gamma_1 \\ \vdots \\ \gamma_M \end{pmatrix}=\mathbf{0} &\iff \sum_{j=1}^M\mathcal{I}_{ij}\cdot\gamma_j=0\qquad \forall i\in\{1,\ldots,N\}\\
		&\iff \sum_{j=1}^M\ C_{v_i,h_j}\cdot\gamma(h_j)=0\qquad \forall i\in\{1,\ldots,N\}\\
		&\iff \gamma \text{ is an eigenfunction of $L^H$ with eigenvalue }0.
		\end{align*}Therefore $\dim(\ker \mathcal{I})=m_H$ and  the proof of the second equality is analogous.
	\end{proof}

	\begin{corollary}\label{cormH}
		\begin{equation*}
		m_H=M-\text{maximum number of linearly independent hyperedges}
		\end{equation*}and
		\begin{equation*}
		m_V=N-\text{maximum number of linearly independent hyperedges}.
		\end{equation*}
	\end{corollary}
	\begin{proof}
		It follows by Proposition \ref{propker} and by the Rank-Nullity Theorem.
	\end{proof}
	\begin{remark}
		In the case when the coefficients are integers, the equation
		\begin{equation}\label{balancingequation}
		\mathcal{I}\cdot\begin{pmatrix} \gamma_1 \\ \vdots \\ \gamma_M \end{pmatrix}=\mathbf{0}
		\end{equation}for the eigenfunctions $\gamma$ of $L^H$ that have eigenvalue $0$, coincides with the \emph{metabolite balancing equation} \cite{steady-state}. In the setting of metabolic pathway analysis, the $v_i$'s represent metabolites, the $h_j$'s represent metabolic reactions, and a vector $\gamma=(\gamma_1,\ldots,\gamma_M)$ is a \emph{flux distribution}. A solution of \eqref{balancingequation} represents a balance between the consumed metabolites and produced metabolites.
		
		Furthermore, in the metabolic pathway  analysis, an \emph{elementary flux mode} (EFM) is defined by Equation (\ref{balancingequation}) together with the following two conditions:
		\begin{enumerate}
			\item \emph{Feasibility}: $\gamma(h)\geq 0$ if the reaction represented by $h\in H$ is irreversible. In our case, this condition is naturally satisfied because, by considering \emph{undirected} hypergraphs, we are representing only \emph{reversible} reactions (at least from the theoretical point of view) as we are considering both orientations for every hyperedge.
			\item \emph{Non-decomposability}: there is no non-zero vector $(\gamma'_1,\ldots,\gamma'_M)\in\mathbb{R}^M$ satisfying Equation (\ref{balancingequation}) and the Feasibility condition such that 
			\begin{equation*}
			P(\gamma):=\{h\in H: \gamma(h)\neq 0\}\supset P(\gamma').
			\end{equation*}This condition is also called \emph{genetic independence} and, in our case, it corresponds to a specific choice of a basis for the kernel of $\mathcal{I}$.
		\end{enumerate}
	\end{remark}
	\section{Min-max principle}\label{Section min-max principle}
	We can apply the Courant-Fischer-Weyl min-max principle in order to characterize all the eigenvalues of $L$ and $L^H$.
	
	\begin{theorem}[Courant--Fischer--Weyl min-max principle]
		Let $W$ be an $n$-dimensional vector space with a positive definite scalar product $(.,.)$. Let $\mathcal{W}_k$ be the family of all $k$-dimensional subspaces of $W$. Let $B : W \rightarrow W$ be a self-adjoint linear operator. Then the eigenvalues $\mu_1\leq\ldots\leq \mu_n$ of $B$ can be obtained by
		\begin{equation*}
		\mu_k=\min_{W_{k}\in\mathcal{W}_{k}}\max_{g(\neq 0)\in W_{k}}\frac{(Bg,g)}{(g,g)}=\max_{W_{n-k+1}\in\mathcal{W}_{n-k+1}}\min_{g(\neq 0)\in W_{n-k+1}}\frac{(Bg,g)}{(g,g)}.
		\end{equation*}
		The vectors $g_k$ realizing such a min-max or max-min then are corresponding eigenvectors, and the min-max spaces $\mathcal{W}_k$ are spanned by the eigenvectors for the eigenvalues $\mu_1,\ldots,\mu_k$, and analogously, the max-min spaces $\mathcal{W}_{n-k+1}$ are spanned by the eigenvectors for the eigenvalues $\mu_k,\ldots,\mu_n$. Thus, we also have
		\begin{equation}\label{eqminmax}
		\mu_k=\min_{g\in W,(g,g_j)=0\text{ for }j=1,\ldots,k-1}\frac{(Bg,g)}{(g,g)}=\max_{g\in V,(g,g_l)=0\text{ for }l=k+1,\ldots,n}\frac{(Bg,g)}{(g,g)}.
		\end{equation}
		In particular,
		\begin{equation*}
		\mu_1=\min_{g\in W}\frac{(Bg,g)}{(g,g)},\qquad \mu_n=\max_{g\in W}\frac{(Bg,g)}{(g,g)}.
		\end{equation*}
	\end{theorem}
	\begin{definition}
		$\frac{(Bg,g)}{(g,g)}$ is the \emph{Rayleigh quotient} of $g$.
	\end{definition}
	\begin{remark}
		Without loss of generality, we may assume $(g,g)=1$ in (\ref{eqminmax}).
	\end{remark}
	
	As a consequence of the min-max principle, we can write the eigenvalues of $L$ and $L^H$ as the \emph{min-max} or \emph{max-min} of the Rayleigh quotients
	\begin{align*}
	\mathrm{RQ}(f):&=\frac{(\delta f,\delta f)_H}{(f,f)_V}\\
	&=\frac{\sum_{h\in H}\biggl(\sum_{v\in V}C_{v,h}\cdot f(v)\biggr)^2}{\sum_{v\in V}\deg v\cdot f(v)^2}, \quad \text{for }f:V\rightarrow\mathbb{R}
	\end{align*}and
	\begin{align*}
	\mathrm{RQ}(\gamma):&=\frac{(\delta^* \gamma,\delta^* \gamma)_V}{(\gamma,\gamma)_H}\\
	&=\frac{\sum_{v\in V}\frac{1}{\deg v}\cdot \biggl(\sum_{h\in H}C_{v,h}\cdot\gamma(h)\biggr)^2}{\sum_{h\in H}\gamma(h)^2}
	,\quad \text{for }\gamma:H\rightarrow\mathbb{R}.
	\end{align*}
	In particular,
	\begin{equation*}
	\lambda_1=\lambda_1^H=\min_{f:V\rightarrow\mathbb{R}}\mathrm{RQ}(f)=\min_{\gamma:H\rightarrow\mathbb{R}}\mathrm{RQ}(\gamma)
	\end{equation*}and
	\begin{equation*}
	\lambda_N=\lambda_M^H=\max_{f:V\rightarrow\mathbb{R}}\mathrm{RQ}(f)=\max_{\gamma:H\rightarrow\mathbb{R}}\mathrm{RQ}(\gamma).
	\end{equation*}
	\section{Bipartiteness and largest eigenvalue}\label{Section Largest eigenvalue}
	A simple graph is \emph{bipartite} if there exists a bipartition of the vertex set into two disjoint sets $V=V_2\sqcup V_2$ such that each edge in is between a vertex in $V_1$ and a vertex in $V_2$. Bipartiteness is an important geometrical property and it is known that the largest eigenvalue of a graph measures how far the graph is from being bipartite. In particular, as shown by Chung in \cite{Chung}, $\lambda_N\leq 2$ for any simple connected graph $\Gamma$, with equality if and only if $\Gamma$ is bipartite.
	
	For chemical hypergraphs the notion of bipartiteness has been generalized in \cite{JM2019}, and in \cite{Sharp} it has been proved that 
	\begin{equation*}
	\lambda_N\leq \max_{h\in H}|h|,
	\end{equation*}with equality if and only if $\Gamma$ is bipartite and $|h|$ is constant for all $h$. Hence, for a chemical hypergraph with constant $|h|$, we can say that $|h|-\lambda_N$ estimates how different the hypergraph is from being bipartite.\newline
	Here we prove a further generalization of the above inequality to the case of hypergraphs with real coefficients. Before, we generalize the definition of bipartite hypergraph.
	\begin{definition}
		Given $h\in H$ and $v\in h$, $v$ is an \emph{input} (resp. \emph{output}) for $h$ if $C_{v,h}>0$ (resp. $C_{v,h}<0$).
	\end{definition}
	
	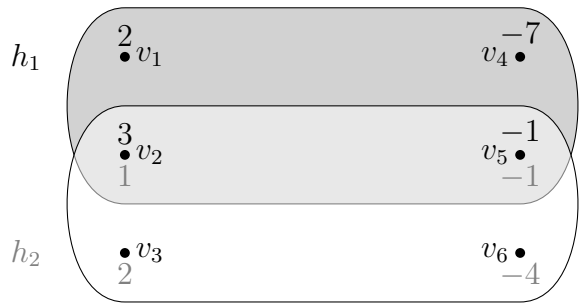
\begin{figure}[t!]
		\begin{center}
			\begin{tikzpicture}[scale=1.3]
			\node (v3) at (1,0) {};
			\node (v2) at (1,1) {};
			\node (v1) at (1,2) {};
			\node (v6) at (5,0) {};
			\node (v5) at (5,1) {};
			\node (v4) at (5,2) {};
			
			\begin{scope}[fill opacity=0.5]
			\filldraw[fill=gray!70] ($(v1)+(0,0.5)$) 
			to[out=180,in=180] ($(v2) + (0,-0.5)$) 
			to[out=0,in=180] ($(v5) + (0,-0.5)$)
			to[out=0,in=0] ($(v4) + (0,0.5)$)
			to[out=180,in=0] ($(v1)+(0,0.5)$);
			\filldraw[fill=white!70] ($(v2)+(0,0.5)$) 
			to[out=180,in=180] ($(v3) + (0,-0.5)$) 
			to[out=0,in=180] ($(v6) + (0,-0.5)$)
			to[out=0,in=0] ($(v5) + (0,0.5)$)
			to[out=180,in=0] ($(v2)+(0,0.5)$);
			\end{scope}
			
			\fill (v1) circle (0.05) node [right] {$v_1$} node [above] {\color{black}$2$};
			\fill (v2) circle (0.05) node [right] {$v_2$} node [above] {\color{black}$3$} node [below] {\color{gray}$1$};
			\fill (v3) circle (0.05) node [right] {$v_3$} node [below] {\color{gray}$2$};
			\fill (v4) circle (0.05) node [left] {$v_4$} node [above] {\color{black}$-7$};
			\fill (v5) circle (0.05) node [left] {$v_5$} node [above] {\color{black}$-1$}node [below] {\color{gray}$-1$};
			\fill (v6) circle (0.05) node [left] {$v_6$} node [below] {\color{gray}$-4$};
			
			\node at (0,2) {\color{black}$h_1$};
			\node at (0,0) {\color{gray}$h_2$};
			\end{tikzpicture}
		\end{center}
		\caption{A bipartite hypergraph with $V_1=\{v_1,v_2,v_3\}$ and $V_2=\{v_4,v_5,v_6\}$. The numbers near the vertices represent the corresponding coefficients.}\label{fig:bipartiteh}
	\end{figure}
	\begin{definition}
		$\Gamma$ is \emph{bipartite} if one can decompose the vertex set as a disjoint union $V=V_1\sqcup V_2$ such that, for every $h\in H$, either $h$ has all its inputs in $V_1$ and all its outputs in $V_2$, or vice versa (Figure \ref{fig:bipartiteh}).
	\end{definition}
	\begin{theorem}\label{thm:lambdaN}
		For every hypergraph with real coefficients $\Gamma$,
		\begin{equation}\label{eq3.1}
		\lambda_N\leq \max_{h\in H}|h|,
		\end{equation}with equality if and only if $\Gamma$ is bipartite, $|h|$ is constant for all $h\in H$ and it is possible to define a non-zero function $f$ on the vertex set such that 
		\begin{equation}\label{eq:gh}
		g(h):=|C_{v,h}\cdot f(v)|
		\end{equation} does not depend on $v$, for all $h\in H$ and $v\in h$. In this case, $f$ satisfying \eqref{eq:gh} is an eigenfunction for $\lambda_N$.
	\end{theorem}
	\begin{proof}
		Let $f:V\rightarrow\mathbb{R}$ be an eigenfunction for $\lambda_N$. By the min-max principle,
		\begin{align*}
		\lambda_N&=\frac{\sum_{h\in H}\Bigl(\sum_{v\in V}C_{v,h}\cdot f(v)\Bigr)^2}{\sum_{v\in V}\deg v f(v)^2}\\
		&\leq \frac{\sum_{h\in H}\Bigl(\sum_{v\in V}|C_{v,h}\cdot f(v)|\Bigr)^2}{\sum_{v\in V}\deg v f(v)^2},
		\end{align*}with equality if and only if $f$ has its non-zero values on a bipartite sub-hypergraph. Now, for each $h\in H$,
		\begin{align*}
		\Bigl( \sum_{v\in V}|C_{v,h}\cdot f(v)|\Bigr)^2&= \sum_{v\in h}C_{v,h}^2\cdot f(v)^2+\sum_{\{v,w\}:\,v\neq w \in h}2\cdot |C_{v,h}\cdot f(v)|\cdot|C_{w,h}\cdot f(w)|\\
		&\leq \sum_{v\in h}C_{v,h}^2\cdot f(v)^2+\sum_{\{v,w\}:\,v\neq w \in h}\biggl(C_{v,h}^2\cdot f(v)^2+C_{w,h}^2\cdot f(w)^2 \biggr)\\
		&=\sum_{v\in h}C_{v,h}^2\cdot f(v)^2 + \sum_{v\in h}(|h|-1)\cdot C_{v,h}^2\cdot f(v)^2\\
		&=|h|\cdot \sum_{v\in h}C_{v,h}^2\cdot f(v)^2,
		\end{align*}with equality if and only if $|C_{v,h}\cdot f(v)|=:g(h)$ is constant for all $v\in h$. Therefore,
		\begin{align*}
		\frac{\sum_{h\in H}\Bigl(\sum_{v\in V}|C_{v,h}\cdot f(v)|\Bigr)^2}{\sum_{v\in V}\deg v f(v)^2}&\leq \frac{\sum_{h\in H}|h|\cdot\sum_{v\in h} C_{v,h}^2\cdot f(v)^2}{\sum_{v\in V}\deg v f(v)^2}\\
		&=\frac{\sum_{v\in V}\sum_{h\ni v}|h|\cdot C_{v,h}^2\cdot f(v)^2}{\sum_{v\in V}\deg v f(v)^2}\\
		&\leq \Bigl(\max_{h\in H}|h|\Bigr)\cdot \frac{\sum_{v\in V}f(v)^2\bigl(\sum_{h\ni v}C_{v,h}^2\bigr)}{\sum_{v\in V}\deg v f(v)^2}\\
		&= \Bigl(\max_{h\in H}|h|\Bigr)\cdot \frac{\sum_{v\in V}\deg v  f(v)^2}{\sum_{v\in V}\deg v f(v)^2}\\
		&=\max_{h\in H}|h|,
		\end{align*}where the first inequality is an equality if and only if
		$g(h)=|C_{v,h}\cdot f(v)|$ does not depend on $v$, for all $h\in H$ and $v\in h$, and the last inequality is an equality if and only if $|h|$ is constant for all $h$. Putting everything together, we have that
		\begin{equation*}
		\lambda_N\leq \max_{h\in H}|h|,
		\end{equation*}with equality if and only if $|h|$ is constant for all $|h|$, $f$ is defined on a bipartite sub-hypergraph and
		\begin{equation}\label{eq:fv}
		f(v)=\frac{g(h)}{|C_{v,h}|},\quad \forall v\in V\text{ and }h\ni v.
		\end{equation}
		Now, in the case when equality is achieved, assume by contradiction that there exists $v'\in V$ with $f(v')=0$. Then, by \eqref{eq:fv} and by the connectivity of $\Gamma$, $f(v)=0$ for all $v\in V$. This is a contradiction. Therefore, if $f$ satisfies \eqref{eq:fv} and it's defined on a bipartite sub-hypergraph, then $\Gamma$ itself is bipartite. This proves the claim.
	\end{proof}
	
	\begin{remark}
		If $\Gamma$ is bipartite, $h$ is constant for all $h\in H$ and $|C_{v,h}|$ is constant for all $v\in V$ and $h\in H$, then clearly it is always possible to find a function $f$ that satisfies \eqref{eq:gh}. Therefore, in this case, by Theorem \ref{thm:lambdaN}, $\lambda_N=\max_{h\in H}|h|$ and $f$ is a corresponding eigenfunction. However, it is not always possible to find a function $f$ that satisfies \eqref{eq:gh}, as shown by the next example.
	\end{remark}
	\begin{example}\label{ex:1}
		\begin{figure}
			\centering
			\includegraphics[width=6cm]{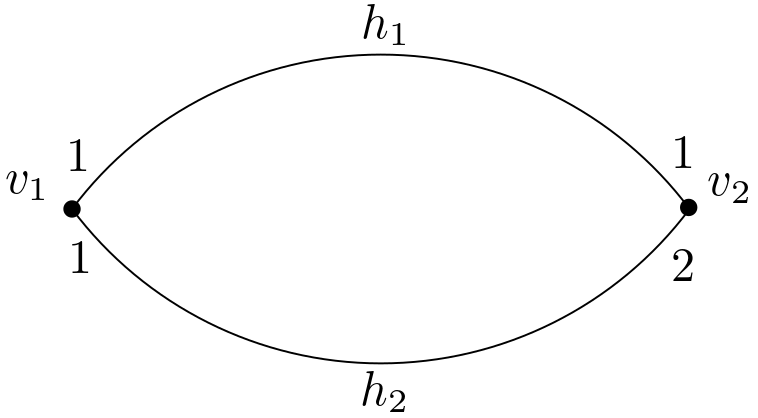}
			\caption{The hypergraph in Example \ref{ex:1}.}\label{fig:ex1}
		\end{figure}
		
		Let $\Gamma=(V,H,\mathcal{C})$ be such that (Figure \ref{fig:ex1}):
		\begin{itemize}
			\item[-] $V=\{v_1,v_2\}$,
			\item[-] $H=\{h_1,h_2\}$,
			\item[-] $C_{v_1,h_1}=C_{v_2,h_1}=C_{v_1,h_2}=1$, and
			\item[-] $C_{v_2,h_2}=2$.
		\end{itemize}Then, $\Gamma$ is a connected bipartite hypergraph such that $|h|=2$ is constant for all $h\in H$. However, it is not possible to find a function $f$ that satisfies \eqref{eq:gh} and therefore, by Theorem \ref{thm:lambdaN}, $\lambda_N=\lambda_2<2$. In fact, a function $f$ satisfying \eqref{eq:gh} in this case should be such that
		\begin{itemize}
			\item[-] $|f(v_1)|=|f(v_2)|=g(h_1)$ and
			\item[-] $|f(v_1)|=2\cdot |f(v_2)|=g(h_2)$,
		\end{itemize}which implies $f(v_1)=f(v_2)=0$, but this is a contraction.
		In this case, in particular, the vertex normalized Laplacian is
		\begin{align*}
		L=&\mathrm{Id}-D^{-1}A\\
		&=\begin{pmatrix}
		\begin{matrix}
		1&0\\
		0&1
		\end{matrix}
		\end{pmatrix}-\begin{pmatrix}
		\begin{matrix}
		1/2&0\\
		0&1/5
		\end{matrix}
		\end{pmatrix}\begin{pmatrix}
		\begin{matrix}
		0&-3\\
		-3&0
		\end{matrix}
		\end{pmatrix}\\
		&=\begin{pmatrix}
		\begin{matrix}
		1&3/2\\
		3/5&1
		\end{matrix}
		\end{pmatrix}.
		\end{align*}Its eigenvalues are
		\begin{equation*}
		\lambda_1=1-\frac{3}{\sqrt{10}}\cong 0.05 \quad \text{and}\quad  \lambda_2=1+\frac{3}{\sqrt{10}}\cong 1.95,
		\end{equation*}with eigenvectors
		\begin{equation*}
		f_1=\bigl(\sqrt{5/2},1\bigr) \quad \text{and}\quad  f_2=\bigl(-\sqrt{5/2},1\bigr),
		\end{equation*}respectively.
	\end{example}
	\section{Bounds on general eigenvalues}\label{Section:general}
	In this section we give some general bounds on all eigenvalues of $\Gamma$.
	
	\begin{definition}
		Given $\hat{v}\in V$, we let $\Gamma\setminus \hat{v}:=(\hat{V},\hat{H},\hat{\mathcal{C}})$, where:
		\begin{itemize}
			\item $\hat{V}=V\setminus\hat{v}$,
			\item $\hat{H}=\{h\setminus\hat{v}:h\in H\}$, and
			\item $\hat{\mathcal{C}}=\{C_{v,h}\in \mathcal{C}:v\in \hat{V} \text{ and }h\in \hat{H}\}$.
		\end{itemize}We say that $\Gamma\setminus \hat{v}$ is obtained from $\Gamma$ by a \emph{weak vertex deletion} of $\hat{v}$. We say that $\Gamma$ is obtained from $\Gamma\setminus \hat{v}$ by a \emph{weak vertex addition} of $\hat{v}$. 
	\end{definition}
	\begin{lemma}\label{lemma:Cauchy}If $\hat{\Gamma}$ is obtained from $\Gamma$ by weak-deleting $r$ vertices,
		\begin{equation*}
		\lambda_{k}(\Gamma)\leq \lambda_k(\hat{\Gamma})\leq \lambda_{k+r}(\Gamma)\qquad\text{for all }k\in\{1,\ldots,n-r\}.
		\end{equation*}
	\end{lemma}
	\begin{proof} As in the proof of Lemma 2.20 in \cite{MulasZhang}, the claim follows directly by inductively applying the Cauchy Interlacing Theorem \cite[Theorem 4.3.17]{MatrixAnalysis}.
	\end{proof}
	\begin{theorem}
		Up to a relabelling of the hyperedges $h_1,\ldots,h_M$, assume that
		\begin{equation*}
		|h_1|\leq \ldots\leq |h_M|.
		\end{equation*}Then,
		\begin{equation}\label{eqN-1}
		\lambda_{N-1}\leq \max\{|h_M|-1,|h_{M-1}|\}
		\end{equation}and, more generally,
		\begin{equation}\label{eqN-k}
		\lambda_{N-k}\leq \max\{|h_M-k|,|h_{M-1}|-k+1,\ldots,|h_{M-k}|\}=\max\{|h_{M-i}|-k+i\}_{i=0,\ldots,k}
		\end{equation}for all $k\in\{1,\ldots,N-1\}$.
	\end{theorem}
	\begin{proof}
		Fix a vertex $v\in h_M$. Then, the largest hyperedge cardinality of $\Gamma\setminus v$ is either $|h_M|-1$ or $|h_{M-1}|$. Therefore, by Theorem \ref{thm:lambdaN} applied to $\Gamma\setminus v$,
		\begin{equation*}
		\lambda_{N-1}(\Gamma\setminus v)\leq  \max\{|h_M|-1,|h_{M-1}|\}.
		\end{equation*}Together with Lemma \ref{lemma:Cauchy}, this proves \eqref{eqN-1}. \eqref{eqN-k} follows by repeatedly applying \eqref{eqN-1}.
	\end{proof}
	\begin{theorem}\label{thm:min}Let $\lambda_{\min}$ be the smallest non-zero eigenvalue of $\Gamma$. Then,
		\begin{equation*}
		\lambda_{\min}\leq \frac{N}{N-m_V}=\frac{N}{M-m_H}\leq \lambda_N
		\end{equation*}and one of the above inequalities is an equality if and only if $\lambda_{\min}=\lambda_N$.
	\end{theorem}
	\begin{proof}
		Same as the proof of Theorem 6.5 in \cite{MulasZhang}.
	\end{proof}
	\begin{remark}
		By Theorem \ref{thm:lambdaN} and Theorem \ref{thm:min} we have that
		\begin{equation*}
		\frac{N}{N-m_V}=\frac{N}{M-m_H}\leq \lambda_N\leq  \max_{h\in H}|h|.
		\end{equation*}Hence
		\begin{equation*}
		m_V\leq N\Biggl(1-\frac{1}{\max_{h\in H}|h|}\Biggr)
		\end{equation*}and
		\begin{equation*}
		m_H\leq M-\frac{N}{\max_{h\in H}|h|}.
		\end{equation*}We can therefore estimate $\max_{h\in H}|h|$ with $m_V$ or $m_H$, and vice versa.
	\end{remark}
	
	\section{Symmetries}\label{Section Symmetries}
	It is known that symmetries can leave signatures in the spectrum of a graph or a hypergraph \cite{Symmetries,symm1,symm2,symm3}. For instance,
	\begin{itemize}
		\item[-] Two vertices in a simple graph are called \emph{duplicate} if they have the same neighbors and it is well known that $\hat{n}$ duplicate vertices produce the eigenvalue $1$ with multiplicity at least $\hat{n}-1$ \cite{duplicationgraphs};
		\item[-] More generally, two vertices in a chemical hypergraph are called \emph{duplicate} if the corresponding rows/columns of the adjacency matrix are the same and it is still true that $\hat{n}$ duplicate vertices produce the eigenvalue $1$ with multiplicity at least $\hat{n}-1$ \cite{MulasZhang};
		\item[-] Two vertices in a chemical hypergraph are called \emph{twin} if they belong exactly to the same hyperedges with the same signs, and $\hat{n}$ twin vertices produce the eigenvalue $0$ with multiplicity at least $\hat{n}-1$ \cite{AndreottiMulas}.
	\end{itemize}
	
	Here we generalize the definitions of twin and anti-twin vertices to the case of hypergraphs with real coefficients, and we show the effect of various hypergraph symmetries in the spectrum of the normalized Laplacian. The propositions in this sections contain results that are new also for the case of chemical hypergraphs. 
	\begin{definition}Two distinct vertices $v_i$ and $v_j$ are:
		\begin{itemize}
			\item[-] \emph{Twin}, if $C_{v_i,h}=C_{v_j,h}$ for all $h\in H$;
			\item[-] \emph{Anti-twin}, if $C_{v_i,h}=-C_{v_j,h}$ for all $h\in H$;
			\item[-] \emph{Duplicate}, if if the corresponding rows/columns of the adjacency matrix are the same;
			\item[-] \emph{Anti-duplicate}, if the corresponding rows (equivalently, columns) of the adjacency matrix have opposite sign.
		\end{itemize}
	\end{definition}
	\begin{remark}
		If $v_i$ and $v_j$ are twin vertices, then
		\begin{equation*}
		\deg v_i=\deg v_j=-A_{ij}
		\end{equation*}while, if they are anti-twin,
		\begin{equation*}
		\deg v_i=\deg v_j=A_{ij}.
		\end{equation*}If $v_i$ and $v_j$ are duplicate or anti-duplicate vertices, their degrees are not necessarily the same and $A_{ij}=A_{ii}=A_{jj}=0$.
	\end{remark}
	
	We need a preliminary remark and a lemma in order to prove the propositions in this section.
	\begin{remark}
		$\lambda$ is an eigenvalue for $L$ and $f:V\rightarrow\mathbb{R}$ is a corresponding eigenfunction if and only if, for each $v_i\in V$, we have that
		\begin{equation*}
		\lambda\cdot f(v_i)=Lf(v_i)=f(v_i)-\frac{1}{\deg v_i}\sum_{j\neq i}A_{ij}f(v_j),
		\end{equation*}which can be rewritten as
		\begin{equation}\label{eq:eigenfunctions}
		(1-\lambda)\cdot f(v_i)=\frac{1}{\deg v_i}\sum_{j\neq i}A_{ij}f(v_j).
		\end{equation}
	\end{remark}
	
	\begin{lemma}\label{lemma:Aij}For all $v_i,v_j\in V$,
		\begin{equation*}
		\deg v_i+\deg v_j\geq \pm 2 A_{ij},
		\end{equation*}with equality if and only if $C_{v_i,h}=\mp C_{v_j,h}$ for all $h\in H$.
	\end{lemma}
	
	\begin{proof}
		We use the fact that, given $a,b\in\mathbb{R}$, $a^2+b^2\geq \mp 2ab$, with equality if and only if $a=\mp b$. This implies that
		\begin{equation*}
		\sum_{h\in H}\bigl(C_{v_i,h}^2+C_{v_j,h}^2\bigr)\geq \mp 2\cdot \sum_{h\in H}C_{v_i,h}\cdot C_{v_j,h},
		\end{equation*}that is,
		\begin{equation*}
		\deg v_i+\deg v_j\geq \pm 2\cdot A_{ij},
		\end{equation*}with equality if and only if $C_{v_i,h}=\mp C_{v_j,h}$ for all $h\in H$.
	\end{proof}
	\begin{proposition}\label{prop:constants}
		The constant functions are eigenfunctions corresponding to some eigenvalue $\lambda$ if and only if $\frac{1}{\deg v_i}\sum_{j\neq i}A_{ij}$ is constant for all $v_i\in V$. In this case,
		\begin{equation*}
		\lambda=1-\frac{1}{\deg v_i}\sum_{j\neq i}A_{ij}.
		\end{equation*}
	\end{proposition}
	\begin{remark}
		In the case of simple graphs, $\sum_{j\neq i}A_{ij}=\deg v_i$ for each $v_i\in V$. Therefore, Proposition \ref{prop:constants} applied to simple graphs states that $\lambda=0$ is always an eigenvalue, and that the constants are corresponding eigenfunctions. This is a well known result.
	\end{remark}
	\begin{remark}
		In the case when $|h|$ is constant for all $h\in H$ and $C_{v,h}=1$ for all $h\in H$ and $v\in H$, the quantity in Proposition \ref{prop:constants} is
		\begin{equation*}
		1-\frac{1}{\deg v_i}\sum_{j\neq i}A_{ij}= 1-\frac{1}{\deg v_i}\sum_{h\ni v_i}(|h|-1)=|h|,
		\end{equation*}for all $v_i\in V$. Proposition \ref{prop:constants} tells us that, in this case, $|h|$ is an eigenvalue with constant eigenfunctions. From \cite[Theorem 3.1]{Sharp}, that we generalized in Theorem \ref{thm:lambdaN}, we know that $\lambda=|h|$ is the largest eigenvalue in this case. Therefore, while for simple graphs the constants are eigenfunctions for $\lambda_1=0$, in the general case the constants can be eigenfunctions for larger eigenvalues.
	\end{remark}

	\begin{proof}[Proof of Proposition \ref{prop:constants}]
		By \eqref{eq:eigenfunctions}, if $f$ is constant then
		\begin{equation}\label{eq:1-lambda}
		1-\lambda=\frac{1}{\deg v_i}\sum_{j\neq i}A_{ij}
		\end{equation}for all $v_i\in V$. Vice versa, if $\frac{1}{\deg v_i}\sum_{j\neq i}A_{ij}$ is constant for all $v_i\in V$, then clearly the constant functions satisfy \eqref{eq:eigenfunctions} with $\lambda=1-\frac{1}{\deg v_i}\sum_{j\neq i}A_{ij}$.
	\end{proof}
	
	\begin{proposition}\label{prop:g11}
		A function $g:V\rightarrow\mathbb{R}$ such that $g(v_i)=g(v_j)=1$ for some $v_i,v_j\in V$ and $g(v_k)=0$ for all $k\neq i,j$ is an eigenfunction for some eigenvalue $\lambda$ if and only if:
		\begin{itemize}
			\item Either $\deg v_i=\deg v_j$ or $A_{ij}=0$;
			\item $A_{ik}=-A_{jk}$ for all $k\neq i,j$.
		\end{itemize}
		In this case,
		\begin{equation*}
		\lambda=1-\frac{A_{ij}}{\deg v_i}\leq 1
		\end{equation*}and $\lambda=0$ if and only if $v_i$ and $v_j$ are anti-twin, while $\lambda=1$ if and only if $v_i$ and $v_j$ are anti-duplicate vertices.
	\end{proposition}
	\begin{proof}
		By \eqref{eq:eigenfunctions}, if $g$ is an eigenfunction for $\lambda$, then
		\begin{itemize}
			\item $(1-\lambda)g(v_i)=(1-\lambda)=\frac{A_{ij}}{\deg v_i}$,
			\item $(1-\lambda)g(v_j)=(1-\lambda)=\frac{A_{ij}}{\deg v_j}$, and
			\item $(1-\lambda)g(v_k)=0=\frac{A_{ik}+A_{jk}}{\deg v_k}$ for all $k\neq i,j$.
		\end{itemize}Therefore, either $\deg v_i=\deg v_j$ or $A_{ij}=0$; $A_{ik}=-A_{jk}$ for all $k\neq i,j$, and $\lambda=1-\frac{A_{ij}}{\deg v_i}$.\newline
		
		Vice versa, if either $\deg v_i=\deg v_j$ or $A_{ij}=0$, and additionally $A_{ik}=-A_{jk}$ for all $k\neq i,j$, it is clear that $g$ satisfies \eqref{eq:eigenfunctions} with $\lambda=1-\frac{A_{ij}}{\deg v_i}$.\newline
		
		In this case, Lemma \ref{lemma:Aij} implies that $A_{ij}\leq \deg v_i$ which, on its turn, implies that $\lambda=1-\frac{A_{ij}}{\deg v_i}\leq 1$. Also, $\lambda=0$ if and only if $A_{ij}= \deg v_i$, therefore (again by Lemma \ref{lemma:Aij}) if and only if $C_{v_i,h}=-C_{v_j,h}$ for all $h\in H$. $\lambda=1$ if and only if $A_{ij}=0$.
	\end{proof}
	
	\begin{proposition}\label{prop:g1-1}
		A function $f:V\rightarrow\mathbb{R}$ such that $f(v_i)=-f(v_j)=1$ for some $v_i,v_j\in V$ and $f(v_k)=0$ for all $k\neq i,j$ is an eigenfunction for some eigenvalue $\lambda$ if and only if:
		\begin{itemize}
			\item Either $\deg v_i=\deg v_j$ or $A_{ij}=0$;
			\item $A_{ik}=A_{jk}$ for all $k\neq i,j$.
		\end{itemize}
		In this case,
		\begin{equation*}
		\lambda=1+\frac{A_{ij}}{\deg v_i}\leq 1
		\end{equation*}and $\lambda=0$ if and only if $v_i$ and $v_j$ are twin, while $\lambda=1$ if and only if $v_i$ and $v_j$ are duplicate vertices.
	\end{proposition}
	\begin{proof}
		Analogous to the proof of Proposition \ref{prop:g11}.
	\end{proof}
	A direct consequence of Proposition \ref{prop:g11} and Proposition \ref{prop:g1-1} is
	\begin{corollary} If $\Gamma$ has $\tilde{n}$ duplicate (or anti-duplicate) vertices, then $1$ is an eigenvalue with multiplicity at least $\tilde{n}-1$.\newline
		If $\Gamma$ has $\hat{n}$ twin (or anti-twin) vertices, then $0$ is an eigenvalue with multiplicity at least $\hat{n}-1$.
	\end{corollary}
	
	We conclude with the example of a hypergraph that satisfies Proposition \ref{prop:g1-1} for two vertices that are neither twin nor duplicate, but nevertheless are strongly symmetric.
	\begin{example}\label{ex:2}
		\begin{figure}
			\centering
			\includegraphics[width=6cm]{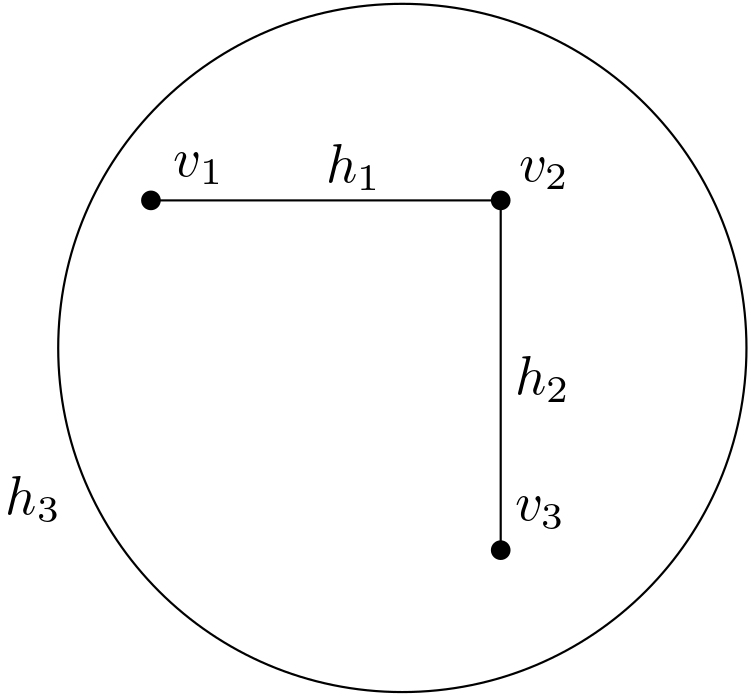}
			\caption{The hypergraph in Example \ref{fig:ex2}.}
			\label{fig:ex2}
		\end{figure}
		Let $\Gamma=(V,H,\mathcal{C})$ be such that (Figure \ref{fig:ex2}):
		\begin{itemize}
			\item[-] $V=\{v_1,v_2,v_3\}$ and $H=\{h_1,h_2\}$;
			\item[-] $h_1=\{v_1,v_2\}$, $h_2=\{v_2,v_3\}$ and $h_3=\{v_1,v_2,v_3\}$;
			\item[-] All non-zero coefficients are equal $1$.
		\end{itemize}Then, $v_1$ and $v_3$ are neither twin nor duplicate vertices, but they satisfy the conditions of Proposition \ref{prop:g1-1} since $\deg v_1=\deg v_3=2$ and $A_{12}=A_{32}=-2,A_{13}=-1$. By Proposition \ref{prop:g1-1}, this symmetry is reflected by the eigenvalue
		\begin{equation*}
		\lambda=1+\frac{A_{13}}{\deg v_1}=\frac{1}{2}. 
		\end{equation*}In particular, in this case
		\begin{align*}
		L=&\mathrm{Id}-D^{-1}A\\
		&=\begin{pmatrix}
		\begin{matrix}
		1&0&0\\
		0&1&0\\
		0&0&1
		\end{matrix}
		\end{pmatrix}-\begin{pmatrix}
		\begin{matrix}
		1/2&0&0\\
		0&1/3&0\\
		0&0&1/2
		\end{matrix}
		\end{pmatrix}\begin{pmatrix}
		\begin{matrix}
		0&-2&-1\\
		-2&0&-2\\
		-1&-2&0
		\end{matrix}
		\end{pmatrix}\\
		&=\begin{pmatrix}
		\begin{matrix}
		1&1&1/2\\
		2/3&1&2/3\\
		1/2&1&1
		\end{matrix}
		\end{pmatrix}.
		\end{align*}Its eigenvalues are
		\begin{equation*}
		\lambda_1=\frac{15-\sqrt{201}}{12}\cong 0.07,\quad \lambda_2=\frac{1}{2} \quad \text{and}\quad  \lambda_3=\frac{15+\sqrt{201}}{12}\cong 2.43.
		\end{equation*}

	\end{example}

	\section{General symmetries}\label{section: general symm}
	
	\begin{definition}
		An \emph{automorphism} of the hypergraph $\Gamma=(V,H,\mathcal{C})$
		consists of bijections $\sigma:V\to V$ and $\sigma:H\to H$ such that
		$\sigma(v)\in \sigma(h)$ iff $v\in h$ and for every $v\in V$ either
		$C_{\sigma(v),\sigma(h)}=C_{v,h}$ for all $h\in H$  or $=-C_{v,h}$ for
		all $h$. We may then put
		$s(\sigma)(v)=1$ in the first and $=-1$ in the second case. When
		$\sigma$ is clear, we simply write $s(v)$ in place of $s(\sigma)(v)$.\\
		For a such an automorphism $\sigma$ and $f:V\rightarrow\mathbb{R}$,
		we put $\sigma_\ast f(v)=s(v)f(\sigma(v))$. 
	\end{definition}
	
	As we shall now verify, the Laplace operator commutes with automorphisms. This
	is trivial if $s(\sigma)(v)=1$ for all $v$. In the general case, we need to be
	careful with the signs. 
	\begin{lemma}\label{commute}
		If $\sigma$ is an automorphism of the hypergraph
		$\Gamma=(V,H,\mathcal{C})$, then
		\begin{equation}
		\label{sym1}
		L(\sigma_\ast f)(v)=\sigma_\ast (Lf)(v)
		\end{equation}
		for all $v\in V$ and $f:V\rightarrow\mathbb{R}$. 
	\end{lemma}
	\begin{proof}
		\begin{eqnarray*}
			L(\sigma_\ast f)(v)&=&   s(v)f(\sigma(v))+ \frac{1}{\deg v}\cdot
			\biggl(\sum_{h\in H}\sum_{v'\in
				V\setminus\{v\}}C_{v,h}\cdot C_{v',h}\cdot
			s(v')f(\sigma(v'))\biggr)\\
			&=& s(v)f(\sigma(v))+ \frac{1}{\deg \sigma(v)}\cdot
			\biggl(\sum_{h\in H}\sum_{v'\in
				V\setminus\{v\}}C_{\sigma(v),\sigma(h)}\cdot C_{\sigma(v'),\sigma(h)}\cdot
			f(\sigma(v'))\biggr)\\
			&=& s(v)f(\sigma(v))+ \frac{1}{\deg \sigma(v)}\cdot
			\biggl(\sum_{h\in H}\sum_{v'\in
				V\setminus\{v\}} \sigma(v)C_{v,h}\cdot
			C_{\sigma(v'),\sigma(h)}\cdot
			f(\sigma(v'))\biggr)\\
			&=& \sigma_\ast (Lf)(v).
		\end{eqnarray*}
	\end{proof}
	
	Thus, the Laplacian $L$ commutes with hypergraph automorphisms $\sigma$.
	We can use Lemma \ref{commute} to decompose the spectrum of $L$.
	Let $\tau$ be an automorphism of the hypergraph
	$\Gamma=(V,H,\mathcal{C})$  with
	\begin{equation}\label{id}
	\tau^2=\mathrm{id}.
	\end{equation}
	Then $\tau$ has two possible eigenvalues, $\pm 1$, on the space of functions
	$f:V\rightarrow\mathbb{R}$, and $L$ leaves those two eigenspaces $L_\pm$
	invariant. Also, $V=V_0+V_1$ where $\tau(v_0)=v_0$ precisely if $v_0\in
	V_0$. That is, $V_0$ is the set of those vertices that are fixed by
	$\tau$. Moreover, we can write $V_1=V'\cup V''$ where $V',V''$ are disjoint and
	$\tau(V')=V''$. Since also $\tau(V'')=V'$ because of \eqref{id}, $V'$ and
	$V''$ play symmetric roles.
	
	W.l.o.g., we assume that $V'$ (and hence also $V''$) is
	connected, as otherwise we can rearrange the decomposition of $V_1$ and/or
	write $\tau$ as the composition of several such automorphisms.
	\begin{lemma}\label{dup}
		In this situation, there is a $|V'|$-dimensional space of functions
		$f:V\rightarrow\mathbb{R}$ consisting of eigenfunctions of $L$ that vanish
		on $V_0$ and that are antisymmetric on $V'$ and $V''$, that is,
		$f(v'')=-s(v')f(v')$ if $v''=\tau(v')\in V''$ for $v'\in V'$, and a remaining $(|V'|+|V_0|)$-dimensional space of eigenfunctions
		that are symmetric on $V'$ and $V''$, that is, $f(v'')=s(v')f(v')$. 
	\end{lemma}
	\begin{proof}
		The first class of functions are those that are eigenfunctions of $\tau$ for
		the eigenvalue $-1$, where the second class has eigenvalue $1$. By Lemma
		\ref{commute}, these are unions of eigenspaces of $L$. Since
		$|V''|=|V'|$ and $V=V_0\cup V' \cup V''$, this generates the space of all
		functions on $V$. 
	\end{proof}
	\begin{definition}
		Let  $\Gamma=(V,H,\mathcal{C})$ be a hypergraph. An \emph{ induced subhypergraph}
		$\hat{\Gamma}$ has  some nonempty vertex set $\hat{V}\subset V$ and a hyperedge
		set $\hat{H}\subset H$ such that
		any two $v_1,v_2 \in \hat{V}$ are contained in  a hyperedge $h\in \hat{H}$
		whenever they are contained in $h$ in $\Gamma$, and the coefficient
		$C_{v,h}$ is then taken from $\Gamma$. An induced subhypergraph is also
		called a \emph{ motif} of $\Gamma$. 
	\end{definition}
	Let $\hat{\Gamma}$ be a motif in  $\Gamma$. We then have the induced Laplacian
	\begin{equation}
	\label{induced}
	L_{\Gamma,\hat{\Gamma}}f(v)=f(v)+\frac{1}{\deg_\Gamma v}\cdot \biggl(\sum_{h\in \hat{H}}\sum_{v'\in \hat{V}\setminus\{v\}}C_{v,h}\cdot C_{v',h}\cdot f(v')\biggr)
	\end{equation}
	where $\deg_\Gamma v$ denotes the degree of $v$ in $\Gamma$.
	\begin{definition}
		We say that the motif $\Gamma'$ with vertex set $V'$ is a \emph{duplicated motif} if $V'$ and $V''$ are disconnected,
		that is, when there is no hyperedge containing elements from both $V'$ and
		$V''$. \\
		We say that $\Gamma'$ and $\Gamma''$ with vertex sets $V'$ and $V''$ are \emph{(anti)twin motives} if for every $h\in H$ we have
		that $v'\in h$ iff $v''=\tau(v')\in h$ and $C_{v',h}=C_{v'',h}$
		($=-C_{v'',h}$) for every $v'\in V'$. 
	\end{definition}
	\begin{lemma}\label{neigh}
		Let $\Gamma'$ be a duplicated motif in $\Gamma$, and let $v_o\in V_o$ be a
		neighbor of some $v'\in V'$. Then $v_o$ is also a neighbor of
		$v''=\tau(v')\in V''$.
	\end{lemma}
	\begin{proof}
		Since $v_0\in V_0$ is fixed by the automorphism $\tau$, and since $\tau$
		maps the hyperedge $h$ containing $v_0$ and $v'$ onto some hyperedge
		$\tau(h)$ containing $v''=\tau(v')$ and $v_0=\tau(v_0)$, the claim follows. 
	\end{proof}
	\begin{lemma}
		Let $\Gamma'$ be a  duplicated motif in $\Gamma$. Then we find a basis of
		eigenfunctions of the Laplacian $L$ of $\Gamma$ of functions $f$ satisfying either
		\begin{enumerate}
			\item
			\begin{equation}
			\label{sym2}
			L_{\Gamma,\Gamma'}f(v)=
			\begin{cases}
			\lambda f(v) & \text{ for } v\in V' \text{ (and also for } v\in V'')\\
			-s(v)f(\tau(v)) &  \text{  for } v\in V''\\
			0 &  \text{ for } v\in V_0
			\end{cases}
			\end{equation}
			\item or
			\begin{equation}
			\label{sym3}
			f(\tau(v))=s(v) f(v) \text{ for } v\in V'.
			\end{equation}
			The latter eigenfunctions are those of the hypergraph $\Gamma^\tau$ obtained
			as the quotient of $\Gamma$ by $\tau$, that is, the hypergraph with vertex
			set $V_0\cup V'$ and all hyperedges induced by $\Gamma$ and the coefficients $C^\tau_{v,h}=\sqrt{2}C_{v,h}$.
		\end{enumerate}
	\end{lemma}
	\begin{proof}If $v_0$ in  $V_0$ and $f(v_0)=0$ and if
		$f$ is antisymmetric, then also $Lf(v_0)=0$, because the contributions from
		its neighbors $v'$ and $v''=\tau(v')$ (see Lemma \ref{neigh}) cancel each
		other in $Lf(v_0)$.
		The result then  follows from Lemma \ref{dup}, because under our assumptions, a neighbor $w$
		of $v\in V'$ is contained either in $V'$ or in $V_0$, in which case for an
		antisymmetric $f$, $f(w)=0$, and therefore, we can restrict the computation
		in \eqref{sym2} to the induced Laplacian, that is, consider only the
		vertices in $V'$. 
	\end{proof}
	Let us consider some {\bf Examples:}
	\begin{enumerate}
		\item Let $v'$ be a duplicated vertex, that is, $V'$ consists only of that
		single vertex. Then the induced Laplacian $L_{\Gamma,\Gamma'}f(v')=f(v')$, and
		therefore, the function $f$ with $f(v')=1, f(v'')=-1, f(v)=0$ for all
		other vertices is an eigenfunction with $\lambda=1$. This was shown in
		\cite{duplicationgraphs} and 
		Prop. \ref{prop:g1-1}.
		\item Consider a graph, and let $V'$ again consist of a single vertex
		$v'$, but assume that $v'$ and $v''=\tau(v')$ are connected by an
		edge. Then \eqref{sym2} becomes
		\begin{equation*}
		f(v')-\frac{1}{\deg_\Gamma v'}f(v'')=\lambda f(v'),
		\end{equation*}
		that is, since $f(v'')=-f(v')$,
		\begin{equation*}
		\lambda =1+\frac{1}{\deg_\Gamma v'}.
		\end{equation*}
		
		\item Let $e=(v_1',v_2')$ be a duplicated edge in a graph. Then we can determine two
		eigenvalues and eigenfunctions of $L$ by solving
		\begin{eqnarray*}
			f(v_1')-\frac{1}{\deg_\Gamma v_1'}f(v_2')&=&\lambda f(v_1')\\
			f(v_2')-\frac{1}{\deg_\Gamma v_2'}f(v_1')&=&\lambda f(v_2')\\
			f(v)=0 & &\text{ for all other }v.
		\end{eqnarray*}
		This yields \cite{duplicationgraphs}
		\begin{equation*}
		\lambda =1\pm \frac{1}{\sqrt{\deg_\Gamma v_1'\deg_\Gamma v_2'}}.
		\end{equation*}
	\end{enumerate}

	\textbf{Acknowledgments.} RM was supported by The Alan Turing Institute under the EPSRC grant EP/N510129/1.

	\bibliographystyle{unsrt}
	\bibliography{Coefficients04.01.21}
\end{document}